\documentclass[11pt,reqno, a4paper]{amsart}
\usepackage{amsfonts}
\usepackage{amssymb}
\usepackage{txfonts}
\usepackage{mathrsfs}
\usepackage{bbm}
\usepackage{enumerate}

\newtheorem{theorem}{Theorem}[section]

\newtheorem{lemma}[theorem]{Lemma}
\newtheorem{proposition}{Proposition}

\theoremstyle{definition}

\newtheorem{remark}{Remark}

 \newcommand{\vstar}{v_{*}}

 \newcommand{\p}{\partial_v}
 
 \newcommand{\s}{\mathcal{S}}
 \newcommand{\Real}{\mathbb{R}}
 
 \newcommand{\Natural}{\mathbb{N}}
 
 \newcommand{\norm}[1]{\Vert#1\Vert}
 
 \newcommand{\abs}[1]{\left\vert#1\right\vert}
 \newcommand{\set}[1]{\left\{#1\right\}}
 \newcommand{\bigset}[1]{\big\{#1\big\}}
 \newcommand{\inner}[1]{\left(#1\right)}
 \newcommand{\biginner}[1]{\big(#1\big)}
 
 \newcommand{\com}[1]{\big[#1\big]}
 \newcommand{\reff}[1]{(\ref{#1})}
 \newcommand{\V}{\,\,|\,\,\,}

\begin{document}
 \title[Propagation of Gevrey Regularity for  Solutions of Landau Equations]
      {Propagation of Gevrey Regularity for  Solutions of Landau Equations}

\author{Hua Chen, Wei-Xi Li \and Chao-Jiang Xu}


 \address{School of Mathematics and Statistics, Wuhan University, Wuhan 430072, P. R. China}
 \email{chenhua@whu.edu.cn}

\address{School of Mathematics and Statistics, Wuhan University,Wuhan 430072, P. R. China}
 \email{weixi.li@yahoo.com}

 \address{School of Mathematics and Statistics, Wuhan University, Wuhan 430072, P. R. China \newline
 \indent and \newline \indent Universit\'{e} de Rouen, UMR 6085-CNRS, Math\'{e}matiques, Avenue de
l'Universit\'{e}, BR.12, \newline \indent F76801 Saint Etienne du
Rouvray, France}
 \email{Chao-Jiang.Xu@univ-rouen.fr}

\thanks{This work is partially supported by the NSFC}
\begin{abstract}
By using the energy-type inequality, we obtain, in this paper, the
result on propagation of Gevrey regularity for the solution of the
spatially homogeneous Landau equation in the cases of Maxwellian
molecules and hard potential.
\end{abstract}

\maketitle

\section{introduction}

There are many papers concerning the propagation of regularity for
the solution of the Boltzmann equation (cf. \cite{DM,DV,Gb,Ga,MV}
and references therein). In these works, it has been shown that the
Sobolev or Lebesgue regularity satisfied by the initial datum is
propagated along the time variable. The solutions having the Gevrey
regularity for a finite time have been constructed in \cite{Ukai84}
in
 which the initial data has the same Gevrey regularity. Recently, the
uniform propagation in all time of the Gevrey regularity has been
proved in \cite{des-fur-ter} in the case of Maxwellian molecules,
which was based on the Wild expansion and the characterization of
the Gevrey regularity by the Fourier transform.

In this paper, we study the propagation of Gevrey regularity for the
solution of Landau equation, which is the limit of the Boltzmann
equation when the collisions become grazing, see \cite{D92} for more
details. Also we know that the Landau equation can be regarded as a
non-linear and non-local analog of the hypo-elliptic Fokker-Planck
equation, and if we choose a suitable orthogonal basis, the Landau
equation in the Maxwellian molecules case will become a non-linear
Fokker-Planck equation (cf. \cite{Villani98}). Recently, a lot of
progress on the Sobolev regularity has been made for the spatially
homogeneous and inhomogeneous Landau equations, cf. \cite{C-D-H, DV,
G, Villani98b} and references therein. On the other hand, in the
Gevrey class frame, the local Gevrey regularity for all variables
$t,x,v$ is obtained in \cite{chen-li-xu} for some semi-linear
Fokker-Planck equations.

Let us consider the following Cauchy problem for  the spatially
homogeneous Landau equation,
\begin{equation}\label{Landau}
\left\{
\begin{array}{ll}
   \partial_t f=\nabla_{v}\cdot\bigset{\int_{\Real^n}
   a(v-\vstar)[f(\vstar)\nabla_v
   f(v)-f(v)\nabla_{v}f(\vstar)]d\vstar},\\
  f(0,v)=f_0(v),
\end{array}
\right.
\end{equation}
where $f(t,v)\geq 0$ stands for the density of particles with
velocity $v\in\Real^n$ at time $t\geq0$, and $(a_{ij})$ is a
nonnegative symmetric matrix given by
\begin{equation}\label{coe}
  a_{ij}(v)=\inner{\delta_{ij}-\frac{v_iv_j}{\abs v^2}}\abs
  v^{\gamma+2}, ~~~~\gamma\in[0,1].
\end{equation}
Here and throughout the paper, we consider only the hard potential
case (i.e. $\gamma\in(0, 1]$) and the Maxwellian molecules case
(i.e. $\gamma=0$).

Set
 \begin{align*}
 b_i=\sum_{j=1}^n\partial_{v_j}a_{ij}=-2\abs v^\gamma v_i, \quad i=1,2,\cdots,n;
 \quad c=\sum_{i,j=1}^n\partial_{v_iv_j}a_{ij}=-2(\gamma+3)\abs
 v^\gamma,\\
 \bar a_{ij}(t,v)=\inner{a_{ij}*f}(t,v)=\int_{\Real^n}a_{ij}(v-\vstar)f(t,\vstar)d\vstar,
 \quad \bar b_i=b_i*f, \quad \bar c=c*f.
 \end{align*}
Then the Cauchy problems \reff{Landau} can be rewritten in the
following form:
\begin{equation}\label{Landau'}
\left\{
  \begin{array}{ll}
  \partial_t f=\sum_{i,j=1}^n
  \bar a_{ij}\partial_{v_iv_j}f-\bar cf,\\
  f(0,v)=f_0(v),
  \end{array}
\right.
\end{equation}
which is a non-linear  diffusion equation with the coefficients
$\bar a_{ij}$ and $\bar c$ depending on the solution $f.$

The motivation for studying the Cauchy problem \reff{Landau'} (cf.
\cite{MX}) comes from the study of the inhomogenous Boltzmann
equations without angular cutoff and non linear Vlasov-Fokker-Planck
equation (see \cite{HN, HF}).

Throughout the paper, for a mult-index
$\alpha=(\alpha_1,\alpha_2,\cdots,\alpha_n)$ and an integer $k$ with
$0\leq k\leq\abs\alpha,$  the notation $D^{\abs\alpha-k}$ is always
used to denote $\p^\gamma$ with the multi-index
 $\gamma$ satisfying $\gamma\leq\alpha$ and
$\abs\gamma=\abs\alpha-k.$   We denote also by $M(f(t)), E(f(t))$
and $H(f(t))$  as the mass, energy and entropy respectively for the
function $f(t,\cdot)$ . That is,
\begin{align*}
M(f(t))=\int_{\Real^n}f(t,v)\,dv,\quad
E(f(t))={1\over2}\int_{\Real^n}f(t,v)\abs v^2\,dv,
\end{align*}
\[H(f(t))=\int_{\Real^n}f(t,v)\log f(t,v)\,dv.\]
Denote here $M_0=M(f(0)), E_0=E(f(0))$ and $H_0=H(f(0)).$  It's
known that the solutions of the Landau equation satisfy the formal
conservation laws:
\[M(f(t))=M_0 ,\quad
E(f(t))=E_0,\quad H(f(t))\leq H_0,\qquad \mbox{for  }\forall ~~t\geq
0.\]

Also in this paper we use the following notations
\[
 \norm {f(t,\cdot)}_{L_s^1}=\int_{\Real^n}f(t,v)\inner{1+\abs v^2}^{s/2}dv,
\]
\[ \norm {\p^\alpha f(t,\cdot)}_{L^2_s}^2=\int_{\Real^n}
 \abs{\p^\alpha f(t,v)}^2\inner{1+\abs v^2}^{s/2}dv,\]
 \[\norm{f(t,\cdot)}_{H^m_s}^2=\sum_{0\leq\abs\alpha\leq m}\norm{\p^\alpha f(t)}_{L^2_s}^2.\]
When there is no risk causing confusion, we write $\norm
{g(t)}_{L_s^1}$ for $\norm {g(t,\cdot)}_{L_s^1}$.

 Next, let us recall the
definition of the Gevrey class function space $G^\sigma(\Real^N),$
where $\sigma\geq 1$ is the Gevrey index (cf. \cite{Rodino93}). Let
$u\in C^\infty(\Real^N)$. We say $u\in G^\sigma(\Real^N)$ if there
exists a constant $C$, called the Gevrey constant,  such that for
all multi-indices $\alpha\in\Natural^N,$
\[
  \norm{\partial^\alpha{u}}_{L^2(\Real^N)}\leq C^{|\alpha|+1}
  (\abs\alpha!)^\sigma.
\]
We denote by $G_0^\sigma(\Real^N)$ the space of Gevrey function with
compact support. Note that $G^1(\Real^N)$ is  space of all real
analytic functions.

 In the hard potential case, the existence, uniqueness
and Sobolev regularity of the weak solution  had been studied in
\cite{DV}, in which they proved that, under suitable assumptions on
the initial datum (e.g. $f_0\in L^1_{2+\delta}$ with $\delta>0,$)
there exists a unique weak solution of the Cauchy problem
\reff{Landau'}, which moreover is in the space
$C^{\infty}\biginner{\Real_t^+; \s(\Real^3)}$. Here $\Real_t^+=(0,
+\infty)$ and $\s$ denotes the space of smooth functions which decay
rapidly at infinity.

 \smallskip
Assuming the existence of  the smooth solution,  we state now the
main result of the paper as follows:

\begin{theorem}\label{main result}
Let $f_0$ be an initial datum with finite mass, energy and entropy.
Suppose $f_0\in G^\sigma(\Real^n)$ with $\sigma>1,$ and  $f$ is a
solution of the Cauchy problem \reff{Landau'} which satisfies
\begin{equation}\label{space}
 f(t,v)\in
  L^\infty_{loc}\inner{[0,+\infty[;~
  H^ m(\Real^n)}\bigcap L_{loc}^2\inner{\,[0,+\infty[;~
  H^{m+1}_\gamma(\Real^n)},\quad \mbox{for all }~m\geq 0.
 \end{equation}
Then $f(t,\cdot)\in G^\sigma(\Real^n)$ for all $t>0$ uniformly,
namely, the Gevrey constant of $f(t,\cdot)$ is independent of $t$.
 More precisely, for any fixed  $T>0$, there exists a constant $C>0$ which is
independent of $t$, such that for any multi-index $\alpha$, one has
\begin{equation*}
 \quad \sup_{t\in[0,T]}\norm{\p^\alpha f(t,\cdot)}_{L^2(\Real^n)}
 \leq
 C^{\abs\alpha+1}\inner{\abs\alpha!}^\sigma.
\end{equation*}

\end{theorem}

\begin{remark}

If $f_0\in L^1_{2+\delta}$ additionally, then by the result of
\cite{DV}, the Cauchy problem \reff{Landau'} admits a solution which
satisfies \reff{space}.

\end{remark}

\begin{remark}

For simplicity, we shall prove Theorem \ref{main result} in the case
of space dimension $n=3.$  The conclusion for  general cases can be
deduced similarly.

\end{remark}

The plan of the paper is as follows: In Section 2 we prove some
lemmas. Section 3 is devoted to the proof of the main result.

\vskip 5ex
\section{Some Lemmas}
\label{section2}

In this section we give some lemmas, which will be used in the proof
of the main result.

\begin{lemma}\label{0806152}
For any $\sigma>1,$ there exists a constant $C_\sigma,$ depending
only on $\sigma,$ such that for all multi-indices $\mu\in\Natural^3,
\abs\mu\geq 1,$
\begin{equation}\label{sum1}
  \sum_{1\leq \abs\beta\leq \abs\mu}\frac{1}{\abs\beta^{3}}\leq
  C_\sigma\abs\mu^{\sigma-1},
\end{equation}
and
\begin{equation}\label{sum2}
  \sum_{1\leq \abs\beta\leq\abs\mu-1}\frac{1}{\abs\beta^2(\abs\mu-\abs\beta)}\leq
  C_\sigma\abs\mu^{\sigma-1}.
\end{equation}
Here and throughout the paper, the notation $\sum_{1\leq
\abs\beta\leq \abs\mu}$ denotes the summation over all the
multi-indices satisfying $\beta\leq\mu$ and $1\leq
\abs\beta\leq\abs\mu$. Also the notation $\sum_{1\leq \abs\beta\leq
\abs\mu-1}$ denotes the summation over all the multi-indices
satisfying $\beta\leq\mu$ and $1\leq \abs\beta\leq\abs\mu-1.$
\end{lemma}

\begin{proof} For each positive integer $l,$ we denote by
$N\set{\abs\beta=l}$ the number of the multi-indices $\beta$ with
$\abs\beta=l.$ In the case when the space dimension is 3, one has
\[N\set{\abs\beta=l}=\frac{(l+2)!}{2!\,l!}={1\over2}(l+1)(l+2).\]
It is easy to deduce that
\begin{align*}
  \sum_{1\leq \abs\beta\leq
  \abs\mu}\frac{1}{\abs\beta^{3}}\leq\sum_{l=1}^{
  \abs\mu}\sum_{\abs\beta=l}\frac{1}{l^{3}}=\sum_{l=1}^{
  \abs\mu}\frac{N\set{\abs\beta=l}}{l^{3}}.
\end{align*}
Combining the estimates above, it holds that
\begin{align*}
  \sum_{1\leq \abs\beta\leq
  \abs\mu}\frac{1}{\abs\beta^{3}}\leq{1\over2}\sum_{l=1}^{
  \abs\mu}\frac{(l+1)(l+2)}{l^{3}}\leq 3\sum_{l=1}^{
  \abs\mu}\frac{1}{l}.
\end{align*}
Observing that $3\sum_{l=1}^{
  \abs\mu}l^{-1}\leq C_\sigma\abs\mu^{\sigma-1}$  for some constant
  $C_\sigma$, we obtain the desired estimate \reff{sum1}.
  Similarly we can deduce the estimate \reff{sum2}.
  \end{proof}

The following lemma is crucial to the proof of Theorem \ref{main
result}.

\begin{lemma}\label{crucial}
Let $\sigma>1$. There exist constant $C_1, C_2>0,$ depending only on
$M_0, E_0, H_0$ and $\gamma,$  such that for all multi-indices
$\mu\in\Natural^3$ with $\abs\mu\geq2$ and all $t>0,$ we have
\begin{align*}
 \partial_t&\norm{\p^\mu f(t)}_{L^2}^2+C_1\norm{\nabla_v\p^\mu  f(t) }_{L^2_\gamma}^2
  \leq C_2\abs\mu^2\norm{\nabla_vD^{\abs\mu-1} f(t)}_{L^2_\gamma}^2\\
  &+C_2 \sum_{2\leq\abs\beta\leq\abs\mu}
  C_{\mu }^\beta \norm{\nabla_vD^{\abs\mu -\abs\beta+1} f(t)}_{L^2_\gamma}
  \cdot\norm{\nabla_vD^{\abs\mu-1}
  f(t)}_{L^2_\gamma}\cdot\com{G_\sigma(f(t))}_{\abs\beta-2}\\
  &+C_2\sum_{0\leq\abs\beta\leq\abs\mu}
   C_{\mu }^\beta\norm{
  \p^\beta f(t)}_{L^2_\gamma}\cdot\norm{\nabla_vD^{\abs\mu-1} f(t)}_{L^2_\gamma}
  \cdot\com{G_\sigma(f(t))}_{\abs{\mu}-\abs\beta},
\end{align*}
where $C_\mu^\beta=\frac{\mu!}{(\mu-\beta)!\beta!}$ is the binomial
coefficient, and
$\com{G_\sigma(f(t))}_{|\nu|}=\norm{\p^{\nu}f(t)}_{L^2}+B^{|\nu|}\inner{|\nu|!}^\sigma$
with $B$ being the constant as given in Lemma \ref{a} below.
\end{lemma}

By the assumption in Theorem \ref{main result},  the solution
$f(t,v)$ of the Cauchy problem \reff{Landau'} is  smooth  in $v$,
and so are the coefficients $\bar a_{ij}=a_{ij}*f$, $\bar
b_{i}=b_{i}*f$ and $\bar c=c*f.$ Here and in what follows, we write
$C$ for a constant, depending only on the Gevrey index $\sigma,$ and
$M_0$, $E_0$ and $H_0$ (the initial mass, energy and entropy), which
may be different in different contexts.

The proof of Lemma \ref{crucial} can be deduced by the following
lemmas:

\begin{lemma}\label{080615}(uniformly ellipticity)
 There exists a constant $K$, depending only on $\gamma$ and $M_0$, $E_0, H_0$, such
 that
 \begin{equation}\label{ellipticity}
  \sum_{i,j=1}^3\bar a_{ij}(t,v)\xi_i\xi_j\geq K(1+\abs
  v^2)^{\gamma/2}\abs\xi^2, \quad \forall~\xi\in\Real^3, \mbox{  and  }\gamma\in[0,1].
 \end{equation}
\end{lemma}

\begin{proof} See Proposition 4 of  \cite{DV} \end{proof}

\begin{remark}
Although the ellipticity of $(\bar a_{ij})$ was proved in \cite{DV}
in the hard potential case $\gamma\in(0,1]$, it still holds for the
Maxwellian case $\gamma=0$. This can be seen in the proof of
Proposition 4 of \cite{DV}.
\end{remark}

\begin{lemma}\label{a}

There exists a constant $B$, depending only on the Gevrey index
$\sigma>1,$ such that for all multi-indices $\beta$ with
$|\beta|\geq 2$ and all $g$, $h\in L^2_\gamma(\Real^3),$ one has
\begin{align*}
&\sum_{i,j=1}^3\int_{\Real^3}(\p^{\beta}\bar a_{ij}(t,v))g(v)h(v)dv
\leq
 C\norm{g}_{L^2_{\gamma}}\norm{h}_{L^2_{\gamma}}\com{G_\sigma(f(t))}_{\abs \beta-2},
 \quad \mbox{for all }\:t>0,
\end{align*}
where $\com{G_\sigma(f(t))}_{\abs \beta-2}=\set{
 \norm{D^{|\beta|-2}f(t)}_{L^2}+B^{\abs{\beta}-2}\com{(|\beta|-2)!}^\sigma}$.
\end{lemma}

\begin{proof} For $\sigma>1,$ there exists a function $\psi\in
G_0^\sigma(\Real^3)$ (cf. \cite{Rodino93}) with compact support in
$\set{v\in\Real^3\V \abs v\leq 2},$ satisfying that $\psi(v)=1$ on
the ball $\set{v\in\Real^3\V \abs v\leq 1}$, and that for some
constant $\tilde B>4$ depending only on $\sigma,$
\begin{equation}\label{psi}
   \sup \abs {\p^{\lambda} \psi}\leq
   \tilde B^{|\lambda|}(|\lambda|-1!)^\sigma,\quad \mbox{for all}~\lambda\in \mathbb {Z} _+^3.
\end{equation}

Write $a_{ij}=\psi a_{ij}+(1-\psi) a_{ij}.$ Then $\bar a_{ij}=(\psi
a_{ij})*f+[(1-\psi) a_{ij}]*f$. It is easy to see that
\[\p^{\beta}\bar
a_{ij}=\com{\p^{\nu}(\psi
a_{ij})}*(\p^{\beta-\nu}f)+\set{\p^{\beta}\com{(1-\psi) a_{ij}}}*f,
~~~\mbox{for  }|\nu|=2.
\]
We first treat the term $\com{\p^{\nu}(\psi
a_{ij})}*(\p^{\beta-\nu}f)$. A direct computation shows that
\begin{align*}
  \abs{\com{\p^{\nu}(\psi
  a_{ij})}*(\p^{\beta-\nu}f)(v)}&=\abs{\int_{\Real^3}
  \com{\p^{\nu}(\psi
  a_{ij})}(v-\vstar)\cdot(\p^{\beta-\nu}f)(\vstar)d\vstar}\\
  &\leq C\int_{\set{\abs{\vstar-v}\leq2}}
  \abs{(\p^{\beta-\nu}f)(\vstar)}d\vstar\\
  &\leq C\norm{\p^{\beta-\nu}f(t)}_{L^2}.
\end{align*}
Next, for the term $\set{\p^{\beta}\com{(1-\psi) a_{ij}}}*f$, one
has, by using the Leibniz's formula,
\begin{align*}
  &\abs{\biginner{\p^{\beta}\com{(1-\psi) a_{ij}}}*f(v)}\\=&
  \big|\sum_{0\leq\abs{\lambda}\leq\abs{\beta}}C_\beta^\lambda\int_{\Real^3}
  \com{\p^{\beta-\lambda}(1-\psi)}(v-\vstar)\cdot\inner{\p^{\lambda}
  a_{ij}}(v-\vstar)\cdot f(\vstar)d\vstar\big|
  \\\leq& \big|\sum_{1\leq\abs{\lambda}\leq\abs{\beta}}C_\beta^\lambda
  \int_{\set{1\leq\abs{\vstar-v}\leq2}}
  \inner{\p^{\beta-\lambda}\psi}(v-\vstar)\cdot\inner{\p^{\lambda}
  a_{ij}}(v-\vstar)\cdot f(\vstar)d\vstar\big|\\
  &+
  \big|\int_{\set{\abs{\vstar-v}\geq1}}
  \com{(1-\psi)}(v-\vstar)\cdot\inner{\p^{\beta}
  a_{ij}}(v-\vstar)\cdot f(\vstar)d\vstar\big|
  \\
  =&J_1+J_2.
\end{align*}
In view of \reff{coe}, we can find a constant $\tilde C,$ depending
only on $\gamma,$ such that
\[
\abs{\inner{\p^{\lambda} a_{ij}}(v-\vstar)}\leq \tilde
C^{\abs\lambda}(\abs\lambda)! \quad {\rm
for}~1\leq\abs{\vstar-v}\leq2.
\]
And for $\abs\beta\geq2,$
\[
\abs{\inner{\p^{\beta} a_{ij}}(v-\vstar)}\leq \tilde
C^{\abs\beta}(\abs\beta)!(1+\abs\vstar^\gamma+\abs v^\gamma) \quad
{\rm for~} 1\leq\abs{\vstar-v}.
\]
From the estimate \reff{psi} we know that $J_1+J_2$ can be estimated
by
\[
 \tilde B^{\abs\beta}(\abs\beta!)^\sigma\cdot\norm{f(t)}_{L^1}
 \sum_{1\leq\abs{\lambda}\leq\abs{\beta}}\biginner{\frac{\tilde
 C}{\tilde B}}^{\abs\lambda}+\tilde
C^{\abs\beta}(\abs\beta!)^\sigma\cdot\norm{f(t)}_{L_\gamma^1}(1+\abs
v^\gamma).
 \]
We can take $ \tilde B$ large enough such that $ \tilde B\geq
2\tilde C.$ Then we get
\begin{align*}
  &\abs{\biginner{\partial_v^{\beta}\com{(1-\psi) a_{ij}}}*f(v)}
  \leq J_1+J_2\\&\leq C\norm{f(t)}_{L_\gamma^1}
  \tilde B^{\abs\beta}(\abs\beta!)^\sigma(1+\abs v^2)^{\gamma/2}\\
  &\leq C \tilde B^{\abs\beta}(\abs\beta!)^\sigma(1+\abs v^2)^{\gamma/2}.
\end{align*}
In the last inequality we used the fact
$\norm{f(t)}_{L_\gamma^1}\leq M_0+2E_0.$ Now we choose a constant
$B$ such that $\tilde B^{\abs\beta}(\abs\beta!)^\sigma\leq
B^{\abs\beta-2}\com{(\abs\beta-2)!}^\sigma.$  It follows immediately
that
\begin{align*}
  \abs{\biginner{\partial_v^{\beta}\com{(1-\psi) a_{ij}}}*f(v)}
  \leq C B^{\abs\beta-2}\com{(\abs\beta-2)!}^\sigma(1+\abs v^2)^{\gamma/2}.
\end{align*}
Combining with the estimate on $\com{\p^{\nu}(\psi
a_{ij})}*(\p^{\beta-\nu}f)$,  we get finally
\begin{align*}
  \abs{\p^{\beta}\bar a_{ij}(v)}&\leq C
  \set{\norm{\p^{\beta-\nu}f(t)}_{L^2}
  +B^{\abs{\beta}-2}\com{(|\beta|-2)!}^\sigma\cdot(1+\abs v^2)^{\gamma/2}}\\
  &\leq C \com{G_\sigma(f(t))}_{\abs{\beta}-2}\cdot
  (1+\abs v^2)^{\gamma/2}.
\end{align*}
Thus, combining with Cauchy's inequality, the estimate above gives
the proof of Lemma \ref{a}.
\end{proof}

Similar to Lemma \ref{a}, we can prove that
\begin{lemma}\label{c}
For all multi-indices $\beta$ with $|\beta|\geq 0$ and all $g,h\in
L^2_\gamma(\Real^3),$  one has
\begin{align*}
&\int_{\Real^3}(\p^{\beta}\bar c(t,v))g(v)h(v)dv \leq
 C\norm{g}_{L^2_\gamma}\norm{h}_{L^2_\gamma}
 \cdot\com{G_\sigma(f(t))}_{\abs \beta}, \quad
\forall\: t\geq0.
\end{align*}
\end{lemma}

Let us now present the proof of the main result of this section.
\begin{proof}[Proof of Lemma \ref{crucial}]
Since
\[\sum_{i=1}^3\partial_{v_i}\bar a_{ij}=\bar b_j,\quad \sum_{j=1}^3\partial_{v_j}b_j=\bar c,\]
and $f$ satisfies $\partial_t f=\sum_{i,j=1}^3\bar
a_{ij}\partial_{v_iv_j}f-\bar c f$, then it holds that
\begin{align*}
 \partial_t\norm{\p^\mu f(t)}_{L^2}^2&=2\int_{\Real^3}\com{\partial_t
 \p^\mu f(t,v)}\cdot\com{\p^\mu  f(t,v)}dv\\&=2\int_{\Real^3}\com{\sum_{i,j=1}^3
 \p^\mu  \biginner{ \bar a_{ij}\partial_{v_iv_j}f-\bar c f}}\cdot\com{\p^\mu  f(t,v)}
 dv.
\end{align*}
Moreover, by using Leibniz's formula, we have
\begin{eqnarray*}
  \partial_t\norm{\p^\mu f(t)}_{L^2}^2&=&2\sum_{i,j=1}^3\int_{\Real^3}
  \bar a_{ij}\inner{\partial_{v_iv_j}\p^\mu  f}\cdot\inner{\p^\mu  f } dv\\
  &&+2\sum_{i,j=1}^3~\sum_{\abs\beta=1}
  C_{\mu }^\beta\int_{\Real^3}\inner{\p^\beta\bar a_{ij}}\inner{\partial_{v_iv_j}\p^{\mu -\beta} f}
  \cdot\inner{\p^\mu f } dv\\
  &&+2\sum_{i,j=1}^3~\sum_{2\leq\abs\beta\leq\abs\mu}
  C_{\mu }^\beta\int_{\Real^3}\inner{\p^\beta\bar a_{ij}}\inner{\partial_{v_iv_j}\p^{\mu -\beta} f}
  \cdot\inner{\p^\mu f } dv\\
  &&-2\sum_{0\leq\abs\beta\leq\abs\mu }
  C_{\mu }^\beta\int_{\Real^3}\inner{\p^{\mu-\beta}\bar c}\inner{
  \p^{\beta} f}\cdot\inner{\p^\mu  f }dv\\
  &=&(I)+(II)+(III)+(IV).
\end{eqnarray*}
Thus the proof of Lemma \ref{crucial} depends on the following
estimates.

\smallskip
{\bf Step 1. Estimate on the term $(I).$}

\smallskip
Integrating by parts, one has
\begin{eqnarray*}
  (I)&=&-2\sum_{i,j=1}^3\int_{\Real^3}
  \bar a_{ij}\inner{\partial_{v_j}\p^\mu f}\cdot\inner{\partial_{v_i}\p^\mu f } dv\\
  &&-2\sum_{j=1}^3\int_{\Real^3}
  \bar b_{j}\inner{\partial_{v_j}\p^\mu f}\cdot\inner{\p^\mu f } dv\\
  &=&(I)_1+(I)_2.
\end{eqnarray*}
For the term $(I)_1$, one has, by applying the ellipticity property
\reff{ellipticity},
\begin{equation*}
  (I)_1\leq -2K\int_{\Real^3}
  \abs{\nabla_v\p^\mu f}^2(1+\abs v^2)^{\gamma/2}dv
  = -2K\norm{\nabla_v\p^\mu f}_{L^2_\gamma}^2.
\end{equation*}
For the term $(I)_2,$ integrating by parts again, we have
\begin{align*}
  (I)_2=-(I)_2+2\int_{\Real^3}
  \bar c\inner{\p^\mu f}\cdot\inner{\p^\mu f } dv.
\end{align*}

Observing that $
  \abs{\bar c(v)}\leq C\norm{f(t)}_{L^1_\gamma}
  (1+\abs v^2)^{\gamma/2}\leq C
  (1+\abs v^2)^{\gamma/2}
$, one has
\[
  (I)_2\leq C\norm{\p^\mu f}_{L^2_\gamma}^2
  \leq C\norm{\nabla_vD^{|\mu|-1} f}_{L^2_\gamma}^2.
\]
This implies
\begin{equation}\label{I}
  (I)\leq -2K\norm{\nabla_v\p^\mu f}_{L^2_{\gamma}}^2+C\norm{\nabla_vD^{\abs\mu-1} f}_{L^2_{\gamma}}^2.
\end{equation}

\smallskip {\bf Step 2. Upper bound for the term $(II).$}

\smallskip
Recall that $(II)=2\sum_{i,j=1}^3\sum_{\abs\beta=1} C_{\mu
}^\beta\int_{\Real^3}\inner{\p^\beta\bar
a_{ij}}\inner{\partial_{v_iv_j}\p^{\mu -\beta} f} \cdot\inner{\p^\mu
f }dv.$ Integrating by parts, we have
\begin{align*}
 (II)&=-2\sum_{j=1}^3\sum_{\abs\beta=1}
  C_{\mu}^\beta\int_{\Real^3}\inner{\p^\beta\bar b_{j}}\inner{\partial_{v_j}\p^{\mu-\beta} f}
  \cdot\inner{\p^\mu f } dv\\
  &-2\sum_{i,j=1}^3~\sum_{\abs\beta=1}
  C_{\mu}^\beta\int_{\Real^3}\inner{\p^\beta\bar a_{ij}}\inner{\partial_{v_j}\p^{\mu-\beta} f}
  \cdot\inner{\partial_{v_i}\p^\mu f } dv\\
  &=(II)_1+(II)_2.
\end{align*}
Observing that $\abs{\p^\beta\bar b_{j}(t,v)}\leq C(1+\abs
v^2)^{\gamma/2}$ for $\abs\beta=1,$ one has
\begin{align*}
 (II)_1\leq C\abs\mu\cdot\norm{\nabla_v\p^{\mu-\beta}f(t)}_{L^2_\gamma}
 \norm{\p^{\mu}f(t)}_{L^2_\gamma}
 \leq
 C\abs\mu\cdot\norm{\nabla_vD^{\abs\mu-1}f(t)}_{L^2_\gamma}^2.
\end{align*}

For the term $(II)_2$, if we write $\mu=\beta+(\mu-\beta),$ then it
holds that
\begin{align*}
 (II)_2&=-2\sum_{i,j=1}^3~\sum_{\abs\beta=1}
  C_{\mu}^\beta\int_{\Real^3}\inner{\p^\beta\bar a_{ij}}\inner{\partial_{v_j}\p^{\mu-\beta} f}
  \cdot\inner{\p^\beta\partial_{v_i}\p^{\mu-\beta} f } dv.
\end{align*}
Since $\abs \beta=1,$ we can integrate by parts to get
\begin{align*}
  (II)_2=&2\sum_{i,j=1}^3~\sum_{\abs\beta=1}
  C_{\mu}^\beta\int_{\Real^3}\inner{\p^{\beta}\bar a_{ij}}\inner{\partial_{v_j}\p^{\mu} f}
  \cdot\inner{\partial_{v_i}\p^{\mu-\beta} f } dv\\
  &+2\sum_{i,j=1}^3~\sum_{\abs\beta=1}
  C_{\mu}^\beta\int_{\Real^3}\inner{\p^{\beta+\beta}\bar a_{ij}}\inner{\partial_{v_j}\p^{\mu-\beta} f}
  \cdot\inner{\partial_{v_i}\p^{\mu-\beta} f } dv\\
  =&-(II)_2+2\sum_{i,j=1}^3~\sum_{\abs\beta=1}
  C_{\mu}^\beta\int_{\Real^3}\inner{\p^{\beta+\beta}\bar a_{ij}}\inner{\partial_{v_j}\p^{\mu-\beta} f}
  \cdot\inner{\partial_{v_i}\p^{\mu-\beta} f } dv.
\end{align*}
Hence
\begin{align*}
  (II)_2&=\sum_{i,j=1}^3~\sum_{\abs\beta=1}
  C_{\mu}^\beta\int_{\Real^3}\inner{\p^{\beta+\beta}\bar a_{ij}}\inner{\partial_{v_j}\p^{\mu-\beta} f}
  \cdot\inner{\partial_{v_i}\p^{\mu-\beta} f } dv.
\end{align*}
Observing  that $\abs{\p^{\beta+\beta}\bar a_{ij}(v)}\leq C (1+\abs
v^2)^{\gamma/2}$ for $|\beta|=1,$  one has
\begin{align*}
 (II)_2\leq C\sum_{\abs\beta=1}
  C_{\mu}^\beta\cdot\norm{\nabla_v\p^{\mu-\beta}f}_{L^2_{\gamma}}^2
  \leq C\abs\mu\cdot\norm{\nabla_vD^{\abs\mu-1}f}_{L^2_{\gamma}}^2,
\end{align*}
which means
\begin{align}\label{II}
 (II)\leq C\abs\mu\cdot\norm{\nabla_vD^{\abs\mu-1}f}_{L^2_{\gamma}}^2.
\end{align}

{\bf Step 3. Upper bound for the term $(III)$ and $(IV)$.}
\smallskip

Recall that $$(III)=2\sum_{i,j=1}^3\sum_{2\leq\abs\beta\leq\abs\mu}
  C_{\mu }^\beta\int_{\Real^3}\inner{\p^\beta\bar a_{ij}}\inner{\partial_{v_iv_j}\p^{\mu -\beta} f}
  \cdot\inner{\p^\mu f }\,dv,$$
and that $$
  (IV)=-2\sum_{0\leq\abs\beta\leq\abs\mu }
  C_{\mu }^\beta\int_{\Real^3}\inner{\p^{\mu -\beta}\bar c}\inner{
  \p^\beta f}\cdot\inner{\p^\mu  f }dv.
$$
By  Lemma \ref{a} and Lemma \ref{c},  we have
\begin{align}\label{III}
\begin{split}
  &(III)\leq C \sum_{i,j=1}^3~\sum_{2\leq\abs\beta\leq\abs\mu}
  C_{\mu }^\beta \norm{\partial_{v_iv_j}\p^{\mu -\beta}f(t)}_{L^2_\gamma}
  \cdot\norm{\p^\mu f(t)}_{L^2_\gamma}\com{G_\sigma(f(t))}_{\abs\beta-2}\\
  &\leq C \sum_{2\leq\abs\beta\leq\abs\mu}
  C_{\mu }^\beta \norm{\nabla_vD^{\abs\mu -\abs\beta+1} f(t)}_{L^2_\gamma}
  \cdot\norm{\nabla_vD^{\abs\mu-1}
  f(t)}_{L^2_\gamma}\com{G_\sigma(f(t))}_{\abs\beta-2},
\end{split}
\end{align}
and
\begin{equation}\label{IV}
  (IV)\leq C\sum_{0\leq\abs\beta\leq\abs\mu}
  C_{\mu }^\beta\norm{
  \p^\beta f(t)}_{L^2_\gamma}\cdot\norm{\nabla_vD^{\abs\mu-1}  f(t) }_{L^2_\gamma}\cdot\com{G_\sigma(f(t))}_{\abs{\mu}-\abs\beta}.
\end{equation}
Combining with the estimates \reff{I}-\reff{IV}, one has
\begin{align*}
  \partial_t&\norm{\p^\mu f(t)}_{L^2}^2+C_1\norm{\nabla_v\p^\mu  f(t) }_{L^2_\gamma}^2
  \leq C_2\abs\mu^2\norm{\nabla_vD^{\abs\mu-1} f(t)}_{L^2_\gamma}^2\\
  &+C_2 \sum_{2\leq\abs\beta\leq\abs\mu}
  C_{\mu }^\beta \norm{\nabla_vD^{\abs\mu -\abs\beta+1} f(t)}_{L^2_\gamma}
  \cdot\norm{\nabla_vD^{\abs\mu-1}
  f(t)}_{L^2_\gamma}\cdot\com{G_\sigma(f(t))}_{\abs\beta-2}\\
  &+C_2\sum_{0\leq\abs\beta\leq\abs\mu}
   C_{\mu }^\beta\norm{
  \p^\beta f(t)}_{L^2_\gamma}\cdot\norm{\nabla_vD^{\abs\mu-1}  f(t) }_{L^2_\gamma}
  \cdot\com{G_\sigma(f(t))}_{\abs{\mu}-\abs\beta}.
\end{align*}
This completes the proof of Lemma \ref{crucial}.
\end{proof}

\vskip 5ex
\section{Proof of Theorem 1.1}
\label{section3} Theorem \ref{main result} will be deduced by the
following result:

\begin{proposition}\label{prp}

Let $\sigma>1$ and $f_0\in  G^{\sigma}(\Real^3)$ be the initial
datum with finite mass, energy and entropy, and let $f$ be a smooth
solution of the Cauchy problem \reff{Landau'} satisfying
\reff{space}. Then for any fixed $T$, $0<T<+\infty,$ there exists a
constant $A$, which depends only on $M_0$, $E_0$, $H_0$, $\gamma$,
$T$, $\sigma$ and the Gevrey constant of $f_0,$ such that for any
$k\in\Natural$, $k\geq 1,$ one has
\begin{equation*}
 (Q)_k \quad \sup_{t\in[0,T]}\norm{\p^\alpha f(t)}_{L^2}
 +\set{\int_{0}^{T}\|\nabla_v \p^{\tilde\alpha}f(t)\|_{L^2_{\gamma}}^2\,dt}^{1/2}\leq
 A^{k}\com{(k-1)!}^\sigma
\end{equation*}
for all multi-indices $\alpha$ and $\tilde\alpha$ with
$\abs\alpha=\abs{\tilde\alpha}=k.$
\end{proposition}
\begin{remark}
From the estimate $(Q)_k$ in Proposition \ref{prp} we can deduce
directly
 the result of Theorem \ref{main result}.
\end{remark}

\begin{proof} [Proof of Proposition \ref{prp}]

We use induction on $k$ to prove the estimate $(Q)_k$. Observe that
$(Q)_1$ holds if we take  $A$  large enough such that
\begin{equation}\label{A}
 A\geq\sup_{t\in[0,T]}\norm {f(t)}_{H^1_\gamma}+ \norm f_{L^2([0, T];
 H^2_\gamma)}+2.
\end{equation}

Assume that the estimate $(Q)_l$ holds for $1\leq l\leq k-1$ and
$k\geq 2$. Then we need to prove that the estimate $(Q)_k$ is true.
Firstly we prove that
\begin{equation}\label{first term}
\sup_{t\in[0,T]}\norm{\p^\alpha f(t)}_{L^2}\leq {1\over
2}A^{\abs\alpha}\com{(\abs\alpha-1)!}^\sigma,\quad \mbox{for all
 }~\abs\alpha=k.
\end{equation}
Applying Lemma \ref{crucial} with $\mu=\alpha,$ we obtain
\begin{align}\label{inter}
\begin{split}
  &\partial_t\norm{\p^\alpha f(t)}_{L^2}^2+C_1\norm{\nabla_v\p^\alpha  f }_{L^2_\gamma}^2
  \leq C_2\abs\alpha^2\norm{\nabla_vD^{\abs\alpha-1} f}_{L^2_\gamma}^2\\
  &+C_2 \sum_{2\leq\abs\beta\leq\abs\alpha}
  C_{\alpha }^\beta \norm{\nabla_vD^{\abs\alpha -\abs\beta+1} f}_{L^2_\gamma}
  \cdot\norm{\nabla_vD^{\abs\alpha-1}
  f}_{L^2_\gamma}\cdot\com{G_\sigma(f(t))}_{\abs\beta-2}\\
  &+C_2\sum_{0\leq\abs\beta\leq\abs\alpha}
   C_{\alpha }^\beta\norm{
  \p^\beta f}_{L^2_\gamma}\cdot\norm{\nabla_vD^{\abs\alpha-1} f}_{L^2_\gamma}
  \cdot\com{G_\sigma(f(t))}_{\abs{\alpha}-\abs\beta}.
\end{split}
\end{align}
Next for the last term on the right hand side of the above
inequality, one has
\begin{align*}
   C&_2\sum_{0\leq\abs\beta\leq\abs\alpha}
   C_{\alpha }^\beta\norm{
  \p^\beta f(t)}_{L^2_\gamma}\cdot\norm{\nabla_v\p^{\abs\alpha-1} f(t)}_{L^2_\gamma}
  \cdot\com{G_\sigma(f(t))}_{\abs{\alpha}-\abs\beta}\\
  =&C_2
   \norm{ f(t)}_{L^2_\gamma}\cdot\norm{\nabla_v\p^{\abs\alpha-1} f(t)}_{L^2_\gamma}
  \cdot\com{G_\sigma(f(t))}_{\abs{\alpha}}\\
  &+C_2\sum_{\abs\beta=1}
   C_{\alpha }^\beta\norm{
  \p^\beta f(t)}_{L^2_\gamma}\cdot\norm{\nabla_v\p^{\abs\alpha-1} f(t)}_{L^2_\gamma}
  \cdot\com{G_\sigma(f(t))}_{\abs{\alpha}-1}\\
  &+C_2\sum_{2\leq\abs\beta\leq\abs\alpha}
   C_{\alpha }^\beta\norm{
  \p^\beta f(t)}_{L^2_\gamma}\cdot\norm{\nabla_v\p^{\abs\alpha-1} f(t)}_{L^2_\gamma}
  \cdot\com{G_\sigma(f(t))}_{\abs{\alpha}-\abs\beta}.
\end{align*}

We denote $\com{G_\sigma(f)}_{\abs\beta}=\sup_{t\in[0,
T]}\com{G_\sigma(f(t))}_{\abs\beta}$. Integrating in both sides of
the estimate \reff{inter} over the interval $[0, T]$, and using the
Cauchy inequality, we get
\begin{align*}
  &\norm{\p^\alpha f(t)}_{L^2}^2-\norm{\p^\alpha
  f(0)}_{L^2}^2
  \leq C_2\abs\alpha^2\int_{0}^{T}\norm{\nabla_v D^{\abs\alpha-1} f(s)}_{L^2_\gamma}^2ds\\
  &+C_2\sum_{2\leq\abs\beta\leq\abs\alpha}
  C_{\alpha}^\beta
  \com{G_\sigma(f)}_{\abs\beta-2}
  \set{\int_{0}^{T}\norm{\nabla_v
  D^{\abs\alpha-\abs\beta+1}f(s)}_{L^2_\gamma}^2ds}^{{1\over2}}\\
  &\qquad\times
  \set{\int_{0}^{T}\norm{\nabla_v D^{\abs\alpha-1}f(s)}_{L^2_\gamma}^2ds}^{{1\over2}}\\
  &+C_2
  \max_{0\leq t\leq T}\norm{f(t)}_{L^2_\gamma}\cdot
  \set{\int_{0}^{T}\com{G(f(s))}_{\abs\alpha}ds}^{{1\over2}}
  \set{\int_{0}^{T}\norm{\nabla_v D^{\abs\alpha-1}f(s)}_{L^2_\gamma}^2ds}^{{1\over2}}\\
  &+C_2\sum_{\abs\beta=1}
  C_{\alpha}^\beta\max_{0\leq t\leq T}\norm{f(t)}_{H^1_\gamma}\cdot
  \com{G_\sigma(f)}_{\abs\alpha-1}
  \cdot\int_{0}^{T}\norm{\nabla_v D^{\abs\alpha-1}f(s)}_{L^2_\gamma}ds\\
  &+C_2\sum_{2\leq\abs\beta\leq\abs\alpha}
  C_{\alpha}^\beta
  \com{G_\sigma(f)}_{\abs\alpha-\abs\beta}\set{\int_{0}^{T}\norm{\p^{\beta}f(s)}_{L^2_\gamma}^2ds}^{{1\over2}}\\
  &\qquad\times
  \set{\int_{0}^{T}\norm{\nabla_v D^{\abs\alpha-1}f(s)}_{L^2_\gamma}^2ds}^{{1\over2}}\\
  &\stackrel{{\rm def}}{=}(S_1)+(S_2)+(S_3)+(S_4)+(S_5).
\end{align*}
From the induction assumption and the fact $\norm {\p^\beta
f}_{L^2_\gamma}\leq \norm {\nabla_v\p^{\abs\beta-1} f}_{L^2_\gamma}
$ for $\abs\beta\geq1$, we have, respectively, the following
estimates:
\begin{equation}\label{ind1}
\set{\int_{0}^{T}\norm{\nabla_v\p^{\abs\alpha-1}f(s)}_{L^2_\gamma}^2ds}^{{1\over2}}
 \leq A^{\abs\alpha-1}\com{(\abs\alpha-2)!}^\sigma;
\end{equation}
\begin{equation}\label{ind2}
\set{\int_{0}^{T}\norm{\nabla\p^{\abs\alpha-\abs\beta+1}f(s)}_{L^2_\gamma}^2ds}^{{1\over2}}
\leq
A^{\abs\alpha-\abs\beta+1}\com{(\abs\alpha-\abs\beta)!}^\sigma,\:
2\leq\abs\beta\leq\abs\alpha;
\end{equation}
\begin{equation}\label{ind3}
\set{\int_{0}^{T}\norm{\p^{\beta}f(s)}_{L^2_\gamma}^2ds}^{{1\over2}}
 \leq
 A^{\abs\beta-1}\com{(\abs\beta-2)!}^\sigma,\quad 2\leq\abs\beta\leq\abs\alpha.
\end{equation}

We next treat the term $\com{G_\sigma(f)}_{m}$ given by
$$
\com{G_\sigma(f)}_{m}=\sup_{t\in[0,T]}\com{G_\sigma(f(t))}_m
=\sup_{t\in[0, T]}\norm{\p^{m}f(t)}_{L^2}+B^{m}\inner{m!}^\sigma,
$$
Observe that for some $\tilde
 C_\sigma>0$,
\[
 B^{m}\inner{m!}^\sigma\leq
 (\tilde C_\sigma B)^{m}\inner{(m-1)!}^\sigma,\quad
 \mbox{for  }1\leq\abs\beta\leq\abs\alpha-1.
\]
Also from the induction assumption, one has
\begin{align*}
\max_{t\in[0, T]}\norm{\p^{m}f(t)}_{L^2}
 \leq A^{m}\inner{(m-1)!}^\sigma,\quad 1\leq m\leq
 \abs\alpha-1=k-1.
\end{align*}
Thus for  $A$ large enough, we have
\begin{align}\label{mfv1}
 \com{G_\sigma(f)}_{m}
 \leq2 A^{m}\inner{(m-1)!}^\sigma, \quad 1\leq m\leq\abs\alpha-1=k-1.
\end{align}

By using the estimate \reff{ind1},  one has
\begin{align}\label{mfv3}
 \int_{0}^{T}\com{G(f(s))}_{\abs\alpha}ds
 \leq C A^{|\alpha|-1}\com{(|\alpha|-2)!}^\sigma.
\end{align}

The estimate for the term $(S_1)$ can be given by \reff{ind1}
directly,
 \begin{align}\label{S1}
 (S_1)&\leq
  C_2\abs\alpha^2\com{A^{\abs\alpha-1}\inner{(\abs\alpha-1)!}^\sigma}^2
   \leq \tilde C_2\com{A^{\abs\alpha-1}\inner{\abs\alpha!}^\sigma}^2 .
\end{align}

We write now the term $(S_2)$ into two parts,
\begin{align*}
  (S_2)=&C_2\sum_{\abs\beta=2}
  C_{\alpha}^\beta
  \com{G_\sigma(f)}_{0}\cdot\int_{0}^{T}\norm{\nabla D^{\abs\alpha-1}f(s)}_{L^2_\gamma}^2ds\\
  &+C_2\sum_{3\leq\abs\beta\leq\abs\alpha}
  C_{\alpha}^\beta
  \com{G_\sigma(f)}_{\abs\beta-2}\set{\int_{0}^{T}\norm{\nabla D^{\abs\alpha-\abs\beta+1}f(s)}_{L^2_\gamma}^2ds}^{1/2}\\
  &\indent\times\set{\int_{0}^{T}\norm{\nabla D^{\abs\alpha-1}f(s)}_{L^2_\gamma}^2ds}^{1/2}\\
  =&(S_2)'+(S_2)^{''}.
\end{align*}
Thus \reff{A} and \reff{ind1} give that
\begin{align*}
 (S_2)'\leq C_2 A \abs\alpha^2\set{
 A^{\abs{\alpha}-1}\com{(\abs\alpha-2)!}^\sigma}^2\leq \tilde C_2 A
 \set{A^{\abs{\alpha}-1}\com{(\abs\alpha-1)!}^\sigma}^2,
\end{align*}
 and \reff{ind1}, \reff{ind2} and \reff{mfv1} give that
\begin{align*}
 (S_2)^{''}&\leq C_2 \sum_{3\leq\abs\beta\leq\abs\alpha}
 \frac{\abs\alpha !}{\abs\beta!(\abs{\alpha}-\abs\beta)!}A^{|\beta|-2}\com{(|\beta|-3)!}^\sigma
 A^{\abs{\alpha}-|\beta|+1}\com{(\abs\alpha-|\beta|)!}^\sigma\\
 &\qquad\times A^{|\alpha|-1}\com{(|\alpha|-2)!}^\sigma.
\end{align*}
Observing  that
\begin{align*}
 \frac{\abs\alpha !}{\abs\beta!(\abs{\alpha}-\abs\beta)!}\com{(|\beta|-3)!}^\sigma
 \com{(\abs\alpha-|\beta|)!}^\sigma
 \leq\frac{6\abs\alpha}{\abs\beta^3}
 \set{(\abs\alpha-1)!}^\sigma,
\end{align*}
we have by the estimate \reff{sum1}
\begin{align*}
 (S_2)^{''}&\leq C_2(A^{\abs{\alpha}-1})^2\com{(\abs\alpha-2)!}^\sigma
 \com{(\abs\alpha-1)!}^\sigma\sum_{3\leq\abs\beta\leq\abs\alpha}\frac{6\abs\alpha}{\abs\beta^3}\\
 &\leq 6C_2(A^{\abs{\alpha}-1})^2\com{(\abs\alpha-2)!}^\sigma
 \com{(\abs\alpha-1)!}^\sigma\abs\alpha^\sigma \\
 &\leq 6C_2\set{A^{\abs{\alpha}-1}
 \com{(\abs\alpha-1)!}^\sigma}^2.
\end{align*}
Therefore
\begin{equation}\label{S2}
 (S_2)=  (S_2)'+(S_2)^{''}\leq 6C_2 A \set{A^{\abs{\alpha}-1}
 \com{(\abs\alpha-1)!}^\sigma}^2.
\end{equation}
Similarly, from the estimates \reff{A}, \reff{ind1} and \reff{mfv3},
one has
\begin{align}\label{S3}
  (S_3)&\leq
   \tilde C_2 A
  \set{A^{\abs\alpha-1}\com{(\abs\alpha-2)!}^\sigma}^2.
\end{align}
From \reff{A}, \reff{ind1} and \reff{mfv1}, one can deduce that
\begin{align}\label{S4}
  (S_4)&\leq
   \tilde C_2 A
  \set{A^{\abs\alpha-1}\com{(\abs\alpha-1)!}^\sigma}^2.
\end{align}

It remains to estimate the term $(S_5).$  From \reff{ind1},
\reff{ind3} and \reff{mfv1} we have
\begin{align*}
 (S_5)\leq &\sum_{2\leq\abs\beta\leq\abs\alpha-1}
 \frac{C_2 \abs\alpha !}{\abs\beta!(\abs{\alpha}-\abs\beta)!}
 A^{\abs{\alpha}-|\beta|}\com{(\abs\alpha-|\beta|-1)!}^\sigma
 A^{|\beta|-1}\com{(|\beta|-2)!}^\sigma\\
 &\qquad\times A^{|\alpha|-1}\com{(|\alpha|-2)!}^\sigma\\
  &+C_2A\set{A^{\abs\alpha-1}\com{(\abs\alpha-2)!}^\sigma}^2.
\end{align*}
Since for $ 2\leq\abs\beta\leq\abs\alpha-1$,
\begin{align*}
 \frac{\abs\alpha !}{\abs\beta!(\abs{\alpha}-\abs\beta)!}
 \com{(\abs\alpha-|\beta|-1)!}^\sigma\com{(|\beta|-2)!}^\sigma
 \leq\frac{2\abs\alpha}{\abs\beta^2(\abs\alpha-\abs\beta)}
 \com{(\abs\alpha-1)!}^\sigma,
\end{align*}
then by using the estimate \reff{sum2}, we have
\begin{align*}
 (S_5)&\leq C_2\inner{A^{|\alpha|-1}}^2\com{(|\alpha|-1)!}^{\sigma}
 \com{(|\alpha|-2)!}^{\sigma}\sum_{2\leq\abs\beta\leq\abs\alpha-1}
 \frac{2\abs\alpha}{\abs\beta^2(\abs\alpha-\abs\beta)}\\
 &\indent+C_2A\set{A^{\abs\alpha-1}\com{(\abs\alpha-2)!}^\sigma}^2\\
 &\leq 3\tilde C_2 A\inner{A^{|\alpha|-1}}^2\com{(|\alpha|-1)!}^{\sigma}
 \com{(|\alpha|-2)!}^{\sigma}\abs\alpha^\sigma\\
 &\leq 3\tilde
 C_2A\set{A^{|\alpha|-1}\com{(|\alpha|-1)!}^{\sigma}}^2.
\end{align*}
This, combined with \reff{S1}-\reff{S4}, implies that
\begin{align*}
  \norm{\p^\alpha
  f(t)}_{L^2}^2
  -\norm{\p^\alpha
  f(0)}_{L^2}^2\leq C_3A
  \set{A^{|\alpha|-1}\com{(|\alpha|)!}^{\sigma}}^2 ,\quad \forall~t\in[0,T].
\end{align*}
Since $f(0)=f_0\in G^\sigma(\Real^3),$ then there exists a constant
$L$ such that
\[\norm{\p^\alpha
  f(0)}_{L^2}^2\leq \set{L^{\abs\alpha}\com{(\abs\alpha-1)!}^\sigma}^2.
\]
Thus taking $A$ large enough, we can deduce that
\[\norm{\p^\alpha
  f(t)}_{L^2}^2\leq \set{{1\over2}A^{\abs\alpha}\com{(\abs\alpha-1)!}^\sigma}^{2},\quad
  \mbox{for all}
  ~t\in[0,T], \mbox{ and }|\alpha|=k.
\]
This gives the proof of the inequality \reff{first term}.

Finally, we need to prove that
\begin{equation}\label{second term}
 \int_ 0^{T}\|\nabla_v \p^{\tilde\alpha}f(t)\|_{L^2_{\gamma}}^2dt\leq
\set{{1\over2}A^{\abs{\tilde\alpha}}\com{(\abs{\tilde\alpha}-1)!}^\sigma}^2,
\qquad \forall~\abs{\tilde\alpha}=k.
\end{equation}
The proof of the estimate \reff{second term} is similar to that of
\reff{first term}. Let us apply Lemma \ref{crucial} again with
$\mu=\tilde\alpha$. Then we have
\begin{align*}
  &\partial_t\norm{\p^{\tilde\alpha} f(t)}_{L^2}^2+C_1\norm{\nabla_v\p^{\tilde\alpha}  f }_{L^2_\gamma}^2
  \leq C_2\abs{\tilde\alpha}^2\norm{\nabla_v\p^{\abs{\tilde\alpha}-1} f}_{L^2_\gamma}^2\\
  &+C_2 \sum_{2\leq\abs\beta\leq\abs{\tilde\alpha}}
  C_{{\tilde\alpha} }^\beta \norm{\nabla_v\p^{\abs{\tilde\alpha} -\abs\beta+1} f}_{L^2_\gamma}
  \cdot\norm{\nabla_v\p^{\abs{\tilde\alpha}-1}
  f}_{L^2_\gamma}\cdot\com{G_\sigma(f(t))}_{\abs\beta-2}\\
  &+C_2\sum_{0\leq\abs\beta\leq\abs{\tilde\alpha}}
   C_{{\tilde\alpha} }^\beta\norm{
  \p^\beta f}_{L^2_\gamma}\cdot\norm{\nabla_v\p^{\abs{\tilde\alpha}-1} f}_{L^2_\gamma}
  \cdot\com{G_\sigma(f(t))}_{\abs{{\tilde\alpha}}-\abs\beta} \\
  &\stackrel{{\rm def}}{=}\mathcal {N}(t).
\end{align*}
Integrating the above inequality over the interval $[0,T]$, we then
have
\[
  C_1\int_{0}^{T}\norm{\nabla_v\p^{\tilde\alpha}f(s)}_{L^2_{\gamma}}^2ds\leq  \norm{\p^{\tilde\alpha}
  f(0)}_{L^2}^2+\int_{0}^{T}\mathcal {N}(s)ds.
\]
By a similar argument as in the proof of \reff{first term},  one has
\[
  \norm{\p^{\tilde\alpha}
  f(0)}_{L^2}^2+ \int_{0}^{T}\mathcal {N}(s)ds\leq C_1
  \set{{1\over2}A^{\abs{\tilde\alpha}}\com{(\abs{\tilde\alpha}-1)!}^\sigma}^2,
\]
which gives the estimate \reff{second term}. The validity of $(Q)_k$
can be derived directly by the estimates in \reff{first term} and
\reff{second term}.
\end{proof}

\medskip

\medskip

\end{document}